\documentclass[amstex,12pt, amssymb]{article}

\usepackage{mathtext}
\usepackage[cp1251]{inputenc}
\usepackage[T2A]{fontenc}
\usepackage[dvips]{graphicx}
\usepackage{amsmath}
\usepackage{amssymb}
\usepackage{amsxtra}
\usepackage{latexsym}
\usepackage{ifthen}

\newcounter{lemma}[section]

\newcounter{corollary}[section]

\newcounter{remark}[section]

\newcounter{theorem}[section]

\newcounter{proposition}[section]

\newcounter{example}

\numberwithin{equation}{section}

\begin{document}

\markboth{\centerline{E.~SEVOST'YANOV}} {\centerline{MAPPINGS WITH
THE INVERSE POLETSKY INEQUALITY...}}

\def\cc{\setcounter{equation}{0}
\setcounter{figure}{0}\setcounter{table}{0}}

\overfullrule=0pt


\author{EVGENY SEVOST'YANOV\\}

\title{
{\bf MAPPINGS WITH THE INVERSE POLETSKY INEQUALITY ARE DISCRETE ON
THE BOUNDARY}}

\date{\today}
\maketitle

\begin{abstract}
It is established a continuous boundary extension of some class of
mappings. Under some additional conditions, we have established that
this extension is light in the closure of the definition domain.
Under some stronger conditions, we also have proved that it is open
and discrete there.
\end{abstract}


\section{Introduction}

The article is devoted to the investigation of mappings with the
so-called inverse Poletsky inequality, which are more general than
quasiregular mappings. In our papers, we have studied their boundary
behavior, see~\cite{SSD} and \cite{Sev$_2$}. In particular, it was
shown that they have a continuous boundary extension while some
majorant in the inequality mentioned above is integrable. Moreover,
it suffices for this majorant to be integrable only on spheres, see
ibid. In this article, we will show something more, namely, that,
under some additional conditions, this extension is light in the
closure. We also prove that, it is discrete in the closure of the
definition domain under some more strong assumptions. We note that
this result cannot be obtained from the previous ones by elementary
reasoning, since its nature is more subtle and requires rather
complicated investigations.

\medskip
Let us turn to the definitions. In what follows, $M_p(\Gamma)$
denotes the {\it $p$-modulus} of a family $\Gamma $
(see~\cite[Section~6]{Va}). We write $M(\Gamma)$ instead
$M_n(\Gamma).$ Let $y_0\in {\Bbb R}^n,$ $0<r_1<r_2<\infty$ and
\begin{equation}\label{eq1**}
A=A(y_0, r_1,r_2)=\left\{ y\,\in\,{\Bbb R}^n:
r_1<|y-y_0|<r_2\right\}\,.\end{equation}
Given $x_0\in{\Bbb R}^n,$ we put
$$B(x_0, r)=\{x\in {\Bbb R}^n: |x-x_0|<r\}\,, \quad {\Bbb B}^n=B(0, 1)\,,$$
$$S(x_0,r) = \{
x\,\in\,{\Bbb R}^n : |x-x_0|=r\}\,. $$
Given sets $E,$ $F\subset\overline{{\Bbb R}^n}$ and a domain
$D\subset {\Bbb R}^n$ we denote by $\Gamma(E,F,D)$ a family of all
paths $\gamma:[a,b]\rightarrow \overline{{\Bbb R}^n}$ such that
$\gamma(a)\in E,\gamma(b)\in\,F$ and $\gamma(t)\in D$ for $t \in [a,
b].$ Given a mapping $f:D\rightarrow {\Bbb R}^n,$ a point $y_0\in
\overline{f(D)}\setminus\{\infty\},$ and
$0<r_1<r_2<r_0=\sup\limits_{y\in f(D)}|y-y_0|,$ we denote by
$\Gamma_f(y_0, r_1, r_2)$ a family of all paths $\gamma$ in $D$ such
that $f(\gamma)\in \Gamma(S(y_0, r_1), S(y_0, r_2),
A(y_0,r_1,r_2)).$ Let $Q:{\Bbb R}^n\rightarrow [0, \infty]$ be a
Lebesgue measurable function and let $p\geqslant 1.$ We say that
{\it $f$ satisfies the inverse Poletsky inequality} at a point
$y_0\in\overline{f(D)}\setminus \{\infty\}$ relative to $p$-modulus,
if the relation

\begin{equation}\label{eq2*B}
M_p(\Gamma_f(y_0, r_1, r_2))\leqslant
\int\limits_{A(y_0,r_1,r_2)\cap f(D)} Q(y)\cdot \eta^p (|y-y_0|)\,
dm(y)
\end{equation}
holds for any Lebesgue measurable function $\eta:
(r_1,r_2)\rightarrow [0,\infty ]$ such that
\begin{equation}\label{eqB2}
\int\limits_{r_1}^{r_2}\eta(r)\, dr\geqslant 1\,.
\end{equation}
Using the inversion $\psi(y)=\frac{y}{|y|^2},$ we also may defined
the relation~(\ref{eq2*B}) at the point $y_0=\infty.$

A mapping $f: D \rightarrow{\Bbb R}^n$ is called {\it discrete} if
the pre-image $\{f^{-1}\left(y\right)\}$ of any point $y\,\in\,{\Bbb
R}^n$ consists of isolated points, and {\it open} if the image of
any open set $U\subset D$ is an open set in ${\Bbb R}^n.$ A mapping
$f$ of $D$ onto $D^{\,\prime}$ is called {\it closed} if $f(E)$ is
closed in $D^{\,\prime}$ for any closed set $E\subset D$ (see, e.g.,
\cite[Chapter~3]{Vu}). Let $h$ be a chordal metric in
$\overline{{\Bbb R}^n},$
$$h(x,\infty)=\frac{1}{\sqrt{1+{|x|}^2}}\,,$$
\begin{equation}\label{eq3C}
h(x,y)=\frac{|x-y|}{\sqrt{1+{|x|}^2} \sqrt{1+{|y|}^2}}\qquad x\ne
\infty\ne y\,.
\end{equation}
and let $h(E):=\sup\limits_{x,y\in E}\,h(x,y)$ be a chordal diameter
of a set~$E\subset \overline{{\Bbb R}^n}$ (see, e.g.,
\cite[Definition~12.1]{Va}).

\medskip
Everywhere further the boundary $\partial A $ of the set $ A $ and
the closure $\overline{A}$ should be understood in the sense
extended Euclidean space $\overline{{\Bbb R}^n}.$ A continuous
extension of the mapping $f:D\rightarrow{\Bbb R}^n$ also should be
understood in terms of mapping with values in $\overline{{\Bbb
R}^n}$ and relative to the metric $h$ in~(\ref{eq3C}) (if a
misunderstanding is impossible). Recall that a domain $D\subset{\Bbb
R}^n$ is called {\it locally connected at the point} $x_0\in
\partial D,$ if for any neighborhood $U$ of a point $x_0$ there is a
neighborhood $V\subset U$ of $x_0$ such that $V\cap D$ is connected.
A domain $D$ is locally connected at $\partial D,$ if $D$ is locally
connected at any point $x_0\in \partial D.$ Similarly,  domain
$D\subset{\Bbb R}^n$ is called {\it finitely connected at the point}
$x_0\in
\partial D,$ if for any neighborhood $U$ of a point $x_0$ there is a
neighborhood $V\subset U$ of $x_0$ such that $V\cap D$ has a finite
number of components. A domain $D$ is finitely connected at
$\partial D,$ if $D$ is finitely connected at any point $x_0\in
\partial D.$

 The boundary of the
domain $D$ is called {\it weakly flat} at the point $x_0\in \partial
D, $ if for any $P> 0$ and for any neighborhood $U$ of a point $x_0
$ there is a neighborhood $V\subset U$ of the same point such that
$M(\Gamma(E, F, D))> P$ for any continua $E, F \subset D,$ which
intersect $\partial U$ and $\partial V.$ The boundary of the domain
$D$ is called weakly flat if the corresponding property is fulfilled
at any point of the boundary $D.$

\medskip
Set
\begin{equation}\label{eq12}
q_{y_0}(r)=\frac{1}{\omega_{n-1}r^{n-1}}\int\limits_{S(y_0,
r)}Q(y)\,d\mathcal{H}^{n-1}(y)\,, \end{equation}
and $\omega_{n-1}$ denotes the area of the unit sphere ${\Bbb
S}^{n-1}$ in ${\Bbb R}^n.$

\medskip
We say that a function ${\varphi}:D\rightarrow{\Bbb R}$ has a {\it
finite mean oscillation} at a point $x_0\in D,$ write $\varphi\in
FMO(x_0),$ if
$$\limsup\limits_{\varepsilon\rightarrow
0}\frac{1}{\Omega_n\varepsilon^n}\int\limits_{B( x_0,\,\varepsilon)}
|{\varphi}(x)-\overline{{\varphi}}_{\varepsilon}|\ dm(x)<\infty\,,
$$
where $\overline{{\varphi}}_{\varepsilon}=\frac{1}
{\Omega_n\varepsilon^n}\int\limits_{B(x_0,\,\varepsilon)}
{\varphi}(x) \,dm(x)$ and $\Omega_n$ is the volume of the unit ball
${\Bbb B}^n$ in ${\Bbb R}^n.$
We also say that a function ${\varphi}:D\rightarrow{\Bbb R}$ has a
finite mean oscillation at $A\subset \overline{D},$ write
${\varphi}\in FMO(A),$ if ${\varphi}$ has a finite mean oscillation
at any point $x_0\in A.$

\medskip
The followings statement is very similar to~Theorem~4 in
\cite{Sev$_2$} or~\cite[Theorem~3.1]{SSD}, but uses some another
conditions on the function~$Q.$ For quasiregular mappings it was
proved in~\cite[Corollary~4.3]{Vu} (cf.~\cite[Theorem~4.2]{Sr}).

\medskip
\begin{theorem}\label{th1}
{\sl\, Let $D\subset {\Bbb R}^n,$ $n\geqslant 2, $ be a domain with
a weakly flat boundary, and let $D^{\,\prime}\subset {\Bbb R}^n$ be
finitely connected at its boundary. Suppose that $f$ is open
discrete and closed mapping of $D$ onto $D^{\,\prime}$ satisfying
the relation~(\ref{eq2*B}) at any point $y_0\in D^{\,\prime}$ with
$p=n.$ Assume that, one of the following conditions hold:

\medskip
1) $Q\in FMO(\partial D^{\,\prime});$

\medskip
2) for any $y_0\in \partial D^{\,\prime}$ there is $\delta(y_0)>0$
such that, for any $0<b_0<\delta(y_0)$ there is $0<a<b_0$ such that
\begin{equation}\label{eq5**}
\int\limits_{a}^{\delta(b_0)}
\frac{dt}{tq_{y_0}^{\frac{1}{n-1}}(t)}>0;
\end{equation}

\medskip
3) for any point $y_0\in \partial D^{\,\prime}$ and
$0<r_1<r_2<r_0:=\sup\limits_{y\in D^{\,\prime}}|y-y_0|$ there is a
set of a positive linear Lebesgue measure $E\subset[r_1, r_2]$ such
that a function $Q$ is integrable on $S(y_0, r)$  for any $r\in E;$

\medskip
4) $Q\in L^1(D).$

\medskip
Then $f$ has a continuous extension
$\overline{f}:\overline{D}\rightarrow\overline{D^{\,\prime}},$ while
$\overline{f}(\overline{D})=\overline{D^{\,\prime}}.$ }
\end{theorem}

\medskip
Let $X$ and $Y$ be metric spaces. Recall that, a mapping
$f:X\rightarrow Y$ is called {\it light,} if for any point $y\in Y,$
the set $f^{\,-1}(y)$ does not contain any nondegenerate continuum
$K\subset X.$ The following statement holds.

\medskip
\begin{theorem}\label{th2}
{\sl\,Let $n\geqslant 2,$ let $p=n$ and let $D$ be a domain with a
weakly flat boundary. Let $f:D\rightarrow D^{\,\prime}$ be an open
discrete and closed mapping of $D$ onto $D^{\,\prime}$ for which
there is a Lebesgue measurable function $Q:{\Bbb R}^n\rightarrow [0,
\infty]$ equals to zero outside $D^{\,\prime}$ such that the
conditions~(\ref{eq2*B})--(\ref{eqB2}) hold for any point $y_0\in
\partial{D^{\,\prime}}.$

\medskip
Assume that, one of the following conditions hold:

\medskip
1) $Q\in FMO(\partial D^{\,\prime});$

\medskip
2) for any $y_0\in \partial D^{\,\prime}$ there is $\delta(y_0)>0$
such that
\begin{equation}\label{eq5D}
\int\limits_{0}^{\delta(y_0)}
\frac{dt}{tq_{y_0}^{\frac{1}{n-1}}(t)}=\infty\,.
\end{equation}

Assume also that, a domain $D^{\,\prime}$ is finitely connected on
its boundary. Then $f$ has a continuous extension
$\overline{f}:\overline{D}\rightarrow \overline{D^{\,\prime}},$
$\overline{f}(\overline{D})=\overline{D^{\,\prime}},$ and it is
light.}
\end{theorem}

\medskip
Given a mapping $f:D\,\rightarrow\,{\Bbb R}^n,$ a set $E\subset D$
and $y\,\in\,{\Bbb R}^n,$ we define the {\it multiplicity function
$N(y,f,E)$} as a number of preimages of the point $y$ in a set $E,$
i.e.
$$
N(y,f,E)\,=\,{\rm card}\,\left\{x\in E: f(x)=y\right\}\,,
$$
\begin{equation}\label{eq1G}
N(f,E)\,=\,\sup\limits_{y\in{\Bbb R}^n}\,N(y,f,E)\,.
\end{equation}
Note that, the concept of a multiplicity function may also be
extended to sets belonging to the closure of $E.$

\medskip
\begin{theorem}\label{th4}
{\sl\, Let $n\geqslant 2,$ let $D$ be a domain with a weakly flat
boundary and let $D^{\,\prime}$ be a domain which is locally
connected on its boundary. Let $f$ be open discrete and closed
mapping of $D$ onto $D^{\,\prime}$ for which there is a Lebesgue
measurable function $Q:{\Bbb R}^n\rightarrow [0, \infty],$ equal to
zero outside $D^{\,\prime},$ such that the
relations~(\ref{eq2*B})--(\ref{eqB2}) with $p=n$ hold at any point
$y_0\in
\partial D^{\,\prime}.$ Assume that, one of the following conditions hold:

\medskip
1) $Q\in FMO(\partial D^{\,\prime});$

\medskip
2) for any $y_0\in \partial D^{\,\prime}$ there is $\delta(y_0)>0$
such that
\begin{equation}\label{eq5B}
\int\limits_{0}^{\delta(y_0)}
\frac{dt}{tq_{y_0}^{\frac{1}{n-1}}(t)}=\infty\,.
\end{equation}
Then $f$ has a continuous extension
$\overline{f}:\overline{D}\rightarrow \overline{D^{\,\prime}}$ such
that $N(f, D)=N(f, \overline{D})<\infty$ and
$\overline{f}(\overline{D})=\overline{D^{\,\prime}}.$ In particular,
$\overline{f}$ is discrete in $\overline{D}.$ }
\end{theorem}

\medskip
Since closed quasiregular mappings satisfy the
relation~(\ref{eq2*B}) with $p=n,$ $Q=K\cdot N(f, D)<\infty$ and
dome $K\geqslant 1$ (see Theorem~3.2 in \cite{MRV}),
Theorem~\ref{th4} immediately implies the following consequence.

\medskip
\begin{corollary}\label{cor2}
{\sl\, Let $n\geqslant 2,$ let $D$ be a domain with a weakly flat
boundary and let $D^{\,\prime}$ be a domain which is locally
connected on its boundary. Let $f$ be open discrete and closed
quasiregular mapping of $D$ onto $D^{\,\prime}.$ Then $f$ has a
continuous extension $\overline{f}:\overline{D}\rightarrow
\overline{D^{\,\prime}}$ such that $N(f, D)=N(f,
\overline{D})<\infty$ and
$\overline{f}(\overline{D})=\overline{D^{\,\prime}}.$ In particular,
$\overline{f}$ is discrete in $\overline{D}.$}
\end{corollary}

\medskip
\begin{remark}
For quasiregular mappings of the unit ball $D={\Bbb B}^n$, the
assertion of Theorem~\ref{th4} was proved in~\cite[Theorem~4.7]{Vu}.
Apparently, the proof of this assertion in the case of an arbitrary
domain has not been given. We also note some cases in which the
assertion of Theorem~\ref{th4} was proved earlier for the case of
homeomorphisms, see~\cite[Theorems~3 and 5]{Sm}
and~\cite[Theorems~4.6 and 13.5]{MRSY}. We also note several rather
important papers concerning continuous extension to the boundary,
lightness, and discreteness of mappings with modulus conditions,
cf.~\cite{Cr$_1$}--\cite{Cr$_3$}.

\end{remark}

\section{Main lemmas on a continuous boundary extension}

Let $D$ be a domain in ${\Bbb R}^n,$ $n\geqslant 2,$ and let $f:D
\rightarrow {\Bbb R}^n$ (or $f:D\rightarrow \overline{{\Bbb R}^n})$
be a discrete open mapping, $\beta: [a,\,b)\rightarrow {\Bbb R}^n$
be a path, and $x\in\,f^{-1}(\beta(a)).$ Recall that, a path
$\alpha: [a,\,b)\rightarrow D$ is called a {\it total $f$-lifting}
of $\beta$ starting at $x,$ if $(1)\quad \alpha(a)=x\,;$ $(2)\quad
(f\circ\alpha)(t)=\beta(t)$ for any $t\in [a, b).$ Note that, any
paths always have total liftings under open discrete and closed
mappings (see \cite[Lermma~3.7]{Vu}).

As usual, the following set is called the {\it locus} of a path
$\gamma: I\rightarrow {\Bbb R}^n:$
$$|\gamma|=\{x\in {\Bbb R}^n: \exists\, t\in [a, b]:
\gamma(t)=x\}\,.$$

The following Lemmas are very close to Theorem~4 in \cite{Sev$_2$},
cf.~\cite[Theorem~3.1]{SSD}. For completeness, we present its proof
in the text of the article in full.

\medskip
\begin{lemma}\label{lem5}{\sl\, Let $D\subset {\Bbb R}^n,$
$n\geqslant 2, $ be a domain with a weakly flat boundary, and let
$D^{\,\prime}\subset {\Bbb R}^n$ be finitely connected at its
boundary. Suppose that $f$ is open discrete and closed mapping of
$D$ onto $D^{\,\prime}$ satisfying the relation~(\ref{eq2*B}) at any
point $y_0\in D^{\,\prime}$ with $p=n.$ In addition, assume that,
for any $z_1\in\partial D^{\,\prime},$ there is
$\varepsilon_0=\varepsilon_0(z_1)> 0$ and a Lebesgue measurable
function $\psi:(0, \varepsilon_0)\rightarrow [0, \infty]$ such that
\begin{equation}\label{eq7B} I(\varepsilon,
\varepsilon_0):=\int\limits_{\varepsilon}^{\varepsilon_0}\psi(t)\,dt
< \infty\quad \forall\,\,\varepsilon\in (0, \varepsilon_0)\,,
\end{equation}
in addition, for any $0<b_0<\varepsilon_0$ there is $0<a\leqslant
b_0$ such that
\begin{equation}\label{eq7D}
I(a, b_0)>0
\end{equation}
besides that,
\begin{equation} \label{eq7C}
\int\limits_{A(z_1, a, b_0)}
Q(x)\cdot\psi^{\,n}(|x-z_1|)\,dm(x)\leqslant C_0I^n(a,
b_0)\,,\end{equation}
where $C_0=C_0(a, b_0)>0$ is some constant which may depend on $a$
and $b_0,$ and $A(z_1, a, b_0)$ is defined in~(\ref{eq1**}). Then
$f$ has a continuous extension
$\overline{f}:\overline{D}\rightarrow\overline{D^{\,\prime}},$ while
$\overline{f}(\overline{D})=\overline{D^{\,\prime}}.$ }
\end{lemma}

\begin{proof}
Assume that the conclusion of Lemma~\ref{lem5} is not true. Then
there are $x_0\in
\partial D$ and at least two sequences $x_i, y_i\in D,$ $i=1,2,\ldots ,$ such
that $x_i, y_i\rightarrow x_0$ as $i\rightarrow\infty,$ while the
relation
\begin{equation}\label{eq1B}
h(f(x_i), f(y_i))\geqslant a>0
\end{equation}
holds for some $a>0$ and all $i\in {\Bbb N},$ where $h$ is a chordal
distance in $\overline{{\Bbb R}^n},$ see~(\ref{eq3C}). Since
$\overline{{\Bbb R}^n}$ is a compact space, we may choose the
sequences $f(x_i)$ and $f(y_i)$ converge to $z_1$ and $z_2$  as
$i\rightarrow\infty $ respectively. Let us assume that
$z_1\ne\infty.$ Since $f$ is closed, it is boundary preserving
(see~\cite[Theorem~3.3]{Vu}), so that $z_1, z_2\in
\partial D^{\,\prime}.$ Since
$D^{\,\prime}$ is finitely connected at its boundary, there are
paths $\alpha:[0, 1)\rightarrow D^{\,\prime}$ and $\beta:[0,
1)\rightarrow D^{\,\prime}$ such that $\alpha\rightarrow z_1$ and
$\beta\rightarrow z_2$ as $t\rightarrow 1-0$ such that $|\alpha|$
contains some subsequence of $f(x_i)$ and $\beta$ contains some
subsequence of $f(y_i),$ $i=1,2,\ldots .$ Without loss of
restriction, using a relabeling, we may assume that $\alpha$ and
$\beta$ contain sequences $f(x_i)$ and $f(y_i),$ respectively. In
addition, we may assume that
\begin{equation}\label{eq8A}
|\alpha|\subset B(z_1, R_*),\, |\beta|\subset {\Bbb R}^n\setminus
B(z_1, R_0)\,,\quad 0<R_*<R_0<\varepsilon_0\,.
\end{equation}
Due to the relation~(\ref{eq7D}), we may assume that
$\int\limits_{R_*}^{R_0}\psi(t)\,dt>0.$
Denote by $\alpha_i$ the subpath of $\alpha$ starting at $f(x_i)$
and ending at $f(x_1),$ and, similarly, denote by $\beta_i$ the
subpath of $\beta$ starting at $f(y_i)$ and ending at $f(y_1).$
Using the change of parameter, we may assume that $\alpha_i:[0,
1]\rightarrow D^{\,\prime}$ and $\beta_i:[0, 1]\rightarrow
D^{\,\prime}.$ Due to~\cite[Lemma~3.7]{Vu}, the paths
$\alpha_i|_{[0, 1)}$ and $\beta_i|_{[0, 1)}$ have total $f$-liftings
starting at points $x_i$ and $y_i,$ respectively. Arguing similarly
to the proof of Lemma~2.1 in \cite{Sev$_3$}, we may show that these
liftings have continuous extensions to $[0, 1].$ Thus, the paths
$\alpha_i$ and $\beta_i$ have total $f$-liftings starting at points
$x_i$ and $y_i,$ as well. Denote these liftings by
$\widetilde{\alpha_i}:[0, 1]\rightarrow D$ and
$\widetilde{\beta_i}:[0, 1]\rightarrow D.$ Observe that, the points
$f(x_1)$ and $f(y_1)$ have at most a finite number of pre-images in
$D$ under $f,$ see~\cite[Theorem~2.8]{MS}. Then there is $r_0>0$
such that $\widetilde{\alpha_i}(1), \widetilde{\beta_i}(1)\in
D\setminus B(x_0, r_0)$ for any $i=1,2,\ldots .$ Since the boundary
of $D$ is weakly flat, for any $P>0$ there is $i=i_P\geqslant 1$
such that the relation
\begin{equation}\label{eq7}
M(\Gamma(|\widetilde{\alpha_i}|, |\widetilde{\beta_i}|,
D))>P\qquad\forall\,\,i\geqslant i_P\,.
\end{equation}
holds. On the other hand, due to~\cite[Theorem~1.I.5.46]{Ku}
\begin{equation}\label{eq9A}
f(\Gamma(|\widetilde{\alpha_i}|, |\widetilde{\beta_i}|,
D))>\Gamma(S(z_1, R_*), S(z_1, R_0), A(z_1, R_*, R_0))\,.
\end{equation}
It follows from~(\ref{eq9A}) that
\begin{equation}\label{eq10B}
\Gamma(|\widetilde{\alpha_i}|, |\widetilde{\beta_i}|, D)
>\Gamma_f(z_0, R_*, R_0)\,.
\end{equation}
In turn, by~(\ref{eq10B}) we have the following:
$$M(\Gamma(|\widetilde{\alpha_i}|, |\widetilde{\beta_i}|, D))$$
\begin{equation}\label{eq11B}
\leqslant M(\Gamma_f(z_1, R_*, R_0))\leqslant \int\limits_{A}
Q(y)\cdot \eta^n (|y-z_1|)\, dm(y)\,,
\end{equation}
where $A=A(z_1, R_*, R_0)$ and $\eta$ is nonnegative Lebesgue
measurable function satisfying the relation~(\ref{eqB2}) for
$r_1:=R_*$ and $r_2:=R_0.$ Put
$$\eta(t)=\left\{
\begin{array}{rr}
\psi(t)/I(R_*, R_0), & t\in (R_*, R_0)\,,\\
0,  &  t\not\in (R_*, R_0)\,,
\end{array}
\right. $$
where $I(R_*, R_0)=\int\limits_{R_*}^{R_0}\,\psi (t)\, dt.$ Observe
that
$\int\limits_{R_*}^{R_0}\eta(t)\,dt=1.$ Now, by the definition of
$f$ in~(\ref{eq2*B}) and due to the relation~(\ref{eq11B}) we obtain
that
\begin{equation}\label{eq14C}
M(\Gamma(|\widetilde{\alpha_i}|, |\widetilde{\beta_i}|, D))\leqslant
C_0<\infty\qquad \forall\quad i\in {\Bbb N}\,.
\end{equation}
The relation~(\ref{eq14C}) contradicts with~(\ref{eq7}). The
resulting contradiction indicates the falsity of the assumption made
in~(\ref{eq1B}).

The proof of the equality
$\overline{f}(\overline{D})=\overline{D^{\,\prime}}$ is similar to
the second part of the proof of Theorem~3.1 in~\cite{SSD}.~$\Box$
\end{proof}

\medskip
\begin{lemma}\label{lem6}{\sl\, Let $D\subset {\Bbb R}^n,$
$n\geqslant 2, $ be a domain with a weakly flat boundary, and let
$D^{\,\prime}\subset {\Bbb R}^n$ be finitely connected at its
boundary. Suppose that $f$ is open discrete and closed mapping of
$D$ onto $D^{\,\prime}$ satisfying the relation~(\ref{eq2*B}) at any
point $y_0\in \partial D^{\,\prime}$ with $p=n.$ In addition, assume
that, for any point $y_0\in D^{\,\prime}$ and
$0<r_1<r_2<r_0:=\sup\limits_{y\in D^{\,\prime}}|y-y_0|$ there is a
set of a positive linear Lebesgue measure $E\subset[r_1, r_2]$ such
that a function $Q$ is integrable on $S(y_0, r)$  for any $r\in E.$
Then $f$ has a continuous extension
$\overline{f}:\overline{D}\rightarrow\overline{D^{\,\prime}},$ while
$\overline{f}(\overline{D})=\overline{D^{\,\prime}}.$ }
\end{lemma}

\begin{proof}
Below we use the standard conventions $a/\infty=0$ for $a\ne\infty,$
$a/0=\infty$ for $a>0$ and $0\cdot\infty=0$ (see, e.g.,
\cite[3.I]{Sa}). Repeating the proof of Lemma~\ref{lem5}, we obtain
the relation~(\ref{eq11B}). Set $\widetilde{Q}(y)=\max\{Q(y), 1\},$
\begin{equation}\label{eq5V}
\widetilde{q}_{y_0}(r)=\frac{1}{\omega_{n-1}r^{n-1}}\int\limits_{S(y_0,
r)}\widetilde{Q}(y)\,d\mathcal{H}^{n-1}(y)
\end{equation}
and
\begin{equation}\label{eq13}
I=\int\limits_{R_*}^{R_0}\frac{dt}{t\widetilde{q}_{z_1}^{1/(n-1)}(t)}\,.
\end{equation}
Observe that $0\ne I\ne \infty.$
Now, the function
$\eta_0(t)=\frac{1}{It\widetilde{q}_{z_1}^{1/(n-1)}(t)}$ satisfies
the relation~(\ref{eqB2}) with $r_1:=R_*$ and $r_2:=R_0.$
Substituting the function $\eta_0$ into the relation~(\ref{eq11B}),
using the Fubini theorem (\cite[Theorem~8.1, Ch.~III]{Sa}) we obtain
that
\begin{equation}\label{eq14A}
M(\Gamma)\leqslant \frac{\omega_{n-1}}{I^{n-1}}<\infty\,.
\end{equation}
The relation~(\ref{eq14A}) contradicts~(\ref{eq7}), which indicates
that the assumption made in~(\ref{eq1B}) is wrong.

The equality $\overline{f}(\overline{D})=\overline{D^{\,\prime}}$ is
proved similarly to the last part of the proof of Theorem~3.1
in~\cite{SSD}.~$\Box$
\end{proof}

\medskip
{\it Proof of Theorem~\ref{th1}.} In the case  1), we choose
$\psi(t)=\frac{1}{t\log\frac{1}{t}}.$ Now, we obtain the
relations~(\ref{eq7B})--(\ref{eq7C}), see
\cite[Proposition~3.1]{Sev$_3$}.

Let us consider the case 2). Set $\widetilde{Q}(y)=\max\{Q(y), 1\}$
and let $\widetilde{q}_{y_0}(r)$ is defined in~(\ref{eq5V}). Let
$$\psi(t)\quad=\quad \left \{\begin{array}{rr}
1/[t\widetilde{q}^{\frac{1}{n-1}}_{y_0}(t)]\ , & \ t\in
(\varepsilon, \varepsilon_0)\ ,
\\ 0\ ,  &  \ t\notin (\varepsilon,
\varepsilon_0)\ ,
\end{array} \right.$$
let $I(a, b):=\int\limits_a^b\psi(t)\,dt.$ Observe that
$$I(\varepsilon, \delta(y_0))
=\int\limits_\varepsilon^{\delta(y_0)}\psi(t)\,dt\leqslant
\int\limits_\varepsilon^{\delta(y_0)}\frac{dt}{t}=\ln\frac{\delta(y_0)}{\varepsilon}<\infty$$
for any $0<\varepsilon<\delta(y_0).$ In addition, given
$0<b_0<\delta(y_0)$ there exists $0<a<b_0$ such that
$$\int\limits_{a}^{b_0}
\frac{dt}{t\widetilde{q}_{y_0}^{\frac{1}{n-1}}(t)}>0$$
due to the condition~(\ref{eq5**}). Finally, applying the Fubini
theorem (\cite[Theorem~8.1, Ch.~III]{Sa}), we obtain that
$$\int\limits_{a<|y-y_0|<b_0}
Q(y)\cdot\psi^n(|y-y_0|)\, dm(y)=\omega_{n-1}\cdot I(a,
b_0)\leqslant C\cdot I^n(a, b_0)\,,$$
where $C=\frac{I(a, b_0)}{I^n(a, b_0)}.$ Thus, the
conditions~(\ref{eq7B})--(\ref{eq7C}) hold. Now, the desired
conclusion follows by Lemma~\ref{lem5}.

\medskip
The case 3) is a conclusion of Lemma~\ref{lem6}. The case 4) is a
particular case of~3), because by the Fubiini theorem,
$$\int\limits_{B(y_0, \varepsilon_0)}Q(y)\,dm(y)=\int\limits_{0}^{\varepsilon_0}
\int\limits_{S(y_0,
t)}Q(y)\,d\mathcal{H}^{n-1}(y)\,dt=\omega_{n-1}\int\limits_{0}^{\varepsilon_0}t^{n-1}
q_{y_0}(t)\,dt<\infty\,.$$
It is worth noting that case~4) was essentially proved
in~\cite[Theorem~3.1]{SSD}, however, for the case when the mapped
domain $D^{\,\prime}$ is locally connected on its boundary. However,
the proof of the new version is not too different from the mentioned
situation.~$\Box$

\medskip
The followings examples show that that the conditions 2)--3) of
Theorem~\ref{th1} do not imply the integrability of the function $Q$
in $D.$ We also have constructed a mapping $f$ corresponding to this
function, cf. Example~1 in~\cite{SevSkv$_3$}.

\medskip
\begin{example}\label{ex1}
Consider the function $\varphi:[0, 1]\rightarrow {\Bbb R}, $ defined
by equality
$$\varphi(t)= \left\{
\begin{array}{rr}
1, & t\in \left(\frac{1}{2k+1}, \frac{1}{2k}\right),\,k=1,2,\ldots\,,\\
\frac{1}{t^2},  &  t\in \left[\frac{1}{2k},\,
\frac{1}{2k-1}\right],\,k=1,2,\ldots\,.
\end{array}
\right. $$
We arbitrarily choose a point $z_0\in
\partial {\Bbb B}^2$ and $z\in {\Bbb B}^2\cap B(z_0, 1),$ and set
\begin{equation}\label{eq10A}Q(z)=\varphi(|z-z_0|)\,,\quad Q:{\Bbb B}^2\cap B(z_0, 1)\setminus\{z_0\}
\rightarrow [0, \infty)\,,
\end{equation}
$$g_1(z)=\frac{z-z_0}{|z-z_0|}\,\rho(|z-z_0|)\,,\qquad g_1(z_0):=0\,,$$
where the function $\rho$ is defined in
\begin{equation}\label{eq3}\rho(r)=
\exp\left\{-\int\limits_{r}^1\frac{dt}{tq_{z_0}(t)}\right\}\,.
\end{equation}
Observe that, $Q(z)$ has finite integrals for almost all spheres
centered at $\partial {\Bbb B}^2$ (it follows from the locally
boundless of $Q$), where, as usual, the function $Q$ equals to zero
outside~${\Bbb B}^2$ and $Q\equiv 1$ in ${\Bbb B}^2\setminus B(z_0,
1).$ For $z\in {\Bbb B}^2\setminus B(z_0, 1)$ we set $f(z)=z.$ Let
us to prove that $g_1$ satisfies the relation,
\begin{equation}\label{eq8B}
M(g_1(\Gamma(S(z_0, r_1), S(z_0, r_2), {\Bbb B}^2)))\leqslant
\int\limits_{{\Bbb B}^2} Q(z)\cdot \eta^2 (|z-z_0|)\, dm(z)\,,
\end{equation}
where $Q(z)=q_{z_0}(|z-z_0|)=\varphi(|z-z_0|)$ (here
$q_{z_0}(|z-z_0|)$ is defined in~(\ref{eq12})). Indeed, $g_1$
belongs to the class $ACL,$ and its Jacobian and the operator norm
of the derivative are calculated by the formulae
\begin{equation*}\label{eq15}\Vert
g_1^{\,\prime}(z)\Vert=\frac{\exp\left\{-\int\limits_{|z-z_0|}^1\frac{dt}{tq_{z_0}(t)}\right\}}{|z-z_0|},\quad
|J(z, g_1)|=\frac{\exp\left\{-2\int
\limits_{|z-z_0|}^1\frac{dt}{tq_{z_0}(t)}\right\}}{|z-z_0|^2q_{z_0}(|z-z_0|)}\,,
\end{equation*}
see e.g.~\cite[Proposition~6.3]{MRSY}. Thus, $g_1\in W_{\rm loc}^{1,
2}({\Bbb B}^2\setminus 0).$ Moreover, the so-called inner dilatation
$K_I(x, g_1)$ of the mapping $g_1$ at $z$ is calculated as follows:
$K_I(z, g_1)=q_{z_0}(|z-z_0|)$ (see, e.g., \cite[Corollary~8.5  and
Theorem~8.6]{MRSY}.

Let us to show that the function $Q$ has infinite integrals over
sufficient small balls $B(z_0, \varepsilon_0).$ For this purpose, we
introduce the polar coordinates $z=(r, \varphi)$ centered at the
point $z_0,$ where $r$ denotes the Euclidean distance from $z_0$ to
$z,$ and $\varphi $ is the angle between the radius vector $z_0-z$
and tangent to the disk ${\Bbb B}^2,$ passing through the point
$z_0.$ Let
$$\theta_1=\inf\limits_{z\in {\Bbb B}^2\cap S(z_0,
\varepsilon)}\varphi,\quad\theta_2=\sup\limits_{z\in {\Bbb B}^2\cap
S(z_0, \varepsilon)}\varphi\,.$$  Using elementary methods of
geometry, we will have that $\sin \theta_1=\varepsilon/2,$ $\sin
(\pi-\theta_2)=\varepsilon/2.$ Then, for $\varepsilon\rightarrow 0,$
we obtain that $\theta_1\rightarrow 0,$ $\theta_2\rightarrow \pi.$
From here it follows that the interval of change of angles $\varphi$
is close to $\pi$ for $z\in {\Bbb B}^2\cap B(z_0, \varepsilon_0).$
In particular, for some (rather small) $\varepsilon_0>0$ we have
that $\theta_2-\theta_1\geqslant 5\pi/6.$ Observe that
$\int\limits_0^{\varepsilon_0}r\varphi(r)\,dr=\infty,$ see
e.g.~\cite[relation~(2.20)]{SevSkv$_3$}. Then, by the Fubini
theorem, we obtain that
$$\int\limits_{{\Bbb B}^2\cap B(z_0,
\varepsilon_0)}Q(z)\,dm(z)=\int\limits_0^{\varepsilon_0}
\int\limits_{S(z_0, r)\cap B(z_0, \varepsilon_0)} Q(z)\,|dz|dr=$$
$$=\int\limits_0^{\varepsilon_0}\int\limits_{\theta_1}^{\theta_2}r\varphi(r)\,dr\geqslant
\frac{5\pi}{6}\cdot\int\limits_0^{\varepsilon_0}r\varphi(r)\,dr=\infty\,.$$
Since $g_1({\Bbb B}^2)$ is simply connected, according to Riemann's
theorem on conformal mapping we may find a mapping $\varphi$ such
that $(\varphi\circ g_1)({\Bbb B}^2)={\Bbb B}^2.$ Put
$f_1:=g_1^{\,-1}\circ \varphi^{\,-1}.$ Then the mapping $f_1$
satisfies all the conditions of Theorem~\ref{th1}
(Lemma~\ref{lem5}), in particular, the inequality~(\ref{eq2*B}) with
$p=n=2.$
\end{example}

\medskip
\begin{example}\label{eqx3}
Finally, let us construct relevant examples of mappings with a
branching. For this purpose, in the notations of Example~\ref{ex1},
we put:
$f_2(z)=(f_1\circ \varphi_1)(z),$ where $\varphi_1(z)=z^2.$ Observe
that, a mapping $f_2$ is open, discrete and closed, in addition, it
satisfies the relation~(\ref{eq2*B}) for $p=n=2$ and $Q:=2K_I(z,
g)=2q_{z_0}(|z-z_0|)$ (see, e.g., \cite[Theorem~3.2]{MRV}).
\end{example}

\section{Auxiliary lemmas}

Following~\cite[Section~2.4]{NP}, we say that a domain $D\subset
{\Bbb R}^n,$ $n\geqslant 2,$ is {\it uniform with respect to
$p$-modulus}, if for any $r>0$ there is $\delta>0$ such that the
inequality
\begin{equation}\label{eq17***}
M_p(\Gamma(F^{\,*},F, D))\geqslant \delta
\end{equation}
holds for any continua $F, F^*\subset D$ with $h(F)\geqslant r$ and
$h(F^{\,*})\geqslant r,$ where $h$ is a chordal metric defined
in~(\ref{eq3C}). When $p=n,$ the prefix ''relative to $p $-modulus''
is omitted. Note that, this definition slightly different from the
''classical'' given in \cite[Chapter~2.4]{NP}, where the sets $F$
and $F^*\subset D $ are assumed to be arbitrary connected. We prove
the following statement, cf.~\cite[Lemma~4.4]{Vu}).

\medskip
\begin{lemma}\label{lem1}
{\sl\, Let $n\geqslant 2,$ $n-1<p\leqslant n,$ and let $D$ be a
domain which is uniform with respect to $p$-modulus. Let
$f:D\rightarrow D^{\,\prime}$ be an open discrete and closed mapping
of $D$ onto $D^{\,\prime}$ for which there is a Lebesgue measurable
function $Q:{\Bbb R}^n\rightarrow [0, \infty],$ equals to zero
outside of $D^{\,\prime},$ such that the
relations~(\ref{eq2*B})--(\ref{eqB2}) hold for some $y_0\in
\overline{D^{\,\prime}}.$ Assume that, there is
$\varepsilon_0=\varepsilon_0(y_0)>0$ and a Lebesgue measurable
function $\psi:(0, \varepsilon_0)\rightarrow [0,\infty]$ such that
$$I(\varepsilon,
\varepsilon_0):=\int\limits_{\varepsilon}^{\varepsilon_0}\psi(t)\,dt
< \infty\quad \forall\,\,\varepsilon\in (0, \varepsilon_0)\,,$$
\begin{equation}\label{eq7***}
I(\varepsilon, \varepsilon_0)\rightarrow
\infty\quad\text{as}\quad\varepsilon\rightarrow 0\,,
\end{equation}
and, in addition,
\begin{equation} \label{eq3.7.2}
\int\limits_{A(y_0, \varepsilon, \varepsilon_0)}
Q(y)\cdot\psi^{\,p}(|y-y_0|)\,dm(y) = o(I^p(\varepsilon,
\varepsilon_0))\end{equation}
as $\varepsilon\rightarrow 0,$ where $A(y_0, \varepsilon,
\varepsilon_0)$ is defined in~(\ref{eq1**}). Let $C_j,$
$j=1,2,\ldots ,$ be a sequence of continua such that
$h(C_j)\geqslant \delta>0$ for some $\delta>0$ and any $j\in {\Bbb
N}$ and, in addition, $h(f(C_j))\rightarrow 0$ as
$j\rightarrow\infty.$ Then $h(f(C_j), y_0)\geqslant \delta_1>0$ for
any $j\in {\Bbb N}$  and some $\delta_1>0.$} \end{lemma}

\begin{proof}
We may assume that $y_0\ne\infty.$ Suppose the opposite, namely, let
$h(f(C_{j_k}), y_0)\rightarrow 0$ as $k\rightarrow\infty$ for some
increasing sequence of numbers $j_k,$ $k=1,2,\ldots .$ Let $F\subset
D$ be any continuum in $D$ such that $y_0\not\in f(F).$ Let
$\Gamma_k:=\Gamma(F, C_{j_k}, D).$ Then, due to the definition of
the uniformity of the domain with respect to $p$-modulus, we obtain
that
\begin{equation}\label{eq3A}
M_p(\Gamma_k)\geqslant \delta_2>0
\end{equation}
for any $k\in {\Bbb N}$ and some $\delta_2>0.$ On the other hand,
let us to consider the family of paths $f(\Gamma_k).$

\medskip
Let us to prove that, for any $l\in {\Bbb N}$ there is a number
$k=k_l$ such that
\begin{equation}\label{eq3B}
f(C_{j_k})\subset B(y_0, 1/l)\,,\qquad k\geqslant k_l\,.
\end{equation}
Suppose the opposite. Then there is $l_0\in {\Bbb N}$ such that
\begin{equation}\label{eq3F}
f(C_{j_{m_l}})\cap ({\Bbb R}^n\setminus B(y_0, 1/l_0))\ne\varnothing
\end{equation}
for some increasing sequence of numbers $m_l,$ $l=1,2,\ldots .$ In
this case, there is a sequence $x_{m_l}\in f(C_{j_{m_l}})\cap ({\Bbb
R}^n\setminus B(y_0, 1/l_0),$ $l\in {\Bbb N}.$ Since by the
assumption $h(f(C_{j_k}), y_0)\rightarrow 0$ for some sequence of
numbers $j_k,$ $k=1,2,\ldots ,$ we obtain that
\begin{equation}\label{eq3E}
h(f(C_{j_{m_l}}), y_0)\rightarrow 0\qquad \text{as}\qquad
l\rightarrow\infty\,.
\end{equation}
Since $h(f(C_{j_{m_l}}), y_0)=\inf\limits_{y\in f(C_{j_{m_l}})}h(y,
y_0)$ and $f(C_{j_{m_l}})$ is a compact as a continuous image of the
compact set $C_{j_{m_l}}$ under the mapping $f,$ it follows that
$h(f(C_{j_{m_l}}), y_0)=h(y_l, y_0),$ where $y_l\in f(C_{j_{m_l}}).$
Due to the relation~(\ref{eq3E}) we obtain that $y_l\rightarrow y_0$
as $l\rightarrow\infty.$ Since by the assumption
$h(f(C_j))=\sup\limits_{y,z\in f(C_j)}h(y,z)\rightarrow 0$ as
$j\rightarrow\infty,$ we have that $h(y_l, x_{m_l})\leqslant
h(f(C_{j_{m_l}}))\rightarrow 0$ as $l\rightarrow\infty.$  Now, by
the triangle inequality, we obtain that
$$h(x_{m_l}, y_0)\leqslant h(x_{m_l}, y_l)+h(y_l, y_0)
\rightarrow 0\qquad \text {as}\quad l\rightarrow\infty\,.$$
The latter contradicts with~(\ref{eq3F}). The contradiction obtained
above proves~(\ref{eq3B}).

\medskip
The following considerations are similar to the second part of the
proof of Lemma~2.1 in~\cite{Sev$_1$}. Without loss of generality we
may consider that the number $l_0\in {\Bbb N}$ is such that
$1/l<\varepsilon_0$ for any $l\geqslant l_0,$ and
\begin{equation}\label{eq3I} f(F)\subset {\Bbb R}^n\setminus B(y_0, 1/l_0)\,.
\end{equation}
In this case, we observe that
\begin{equation}\label{eq3G}
f(\Gamma_{k_l})>\Gamma(S(y_0, 1/l), S(y_0, 1/l_0), A(y_0,1/l,
1/l_0))\,.
\end{equation}
Indeed, let $\widetilde{\gamma}\in f(\Gamma_{k_l}).$ Then
$\widetilde{\gamma}(t)=f(\gamma(t)),$ where $\gamma\in
\Gamma_{k_l},$ $\gamma:[0, 1]\rightarrow D,$ $\gamma(0)\in F,$
$\gamma(1)\in C_{j_{k_l}}.$ Due to the relation~(\ref{eq3I}), we
obtain that $f(\gamma(0))\in f(F)\subset {\Bbb R}^n\setminus B(y_0,
1/l_0).$ In addition, by~(\ref{eq3B}) we have that $\gamma(1)\in
C_{j_{k_l}}\subset B(y_0, 1/l_0).$ Thus, $|f(\gamma(t))|\cap B(y_0,
1/l_0)\ne\varnothing \ne |f(\gamma(t))|\cap ({\Bbb R}^n\setminus
B(y_0, 1/l_0)).$ Now, by~\cite[Theorem~1.I.5.46]{Ku} we obtain that,
there is $0<t_1<1$ such that $f(\gamma(t_1))\in S(y_0, 1/l_0).$ Set
$\gamma_1:=\gamma|_{[t_1, 1]}.$ We may consider that
$f(\gamma(t))\in B(y_0, 1/l_0)$ for any $t\geqslant t_1.$ Arguing
similarly, we obtain $t_2\in [t_1, 1]$ such that $f(\gamma(t_2))\in
S(y_0, 1/l).$ Put $\gamma_2:=\gamma|_{[t_1, t_2]}.$ We may consider
that $f(\gamma(t))\not\in B(y_0, 1/l)$ for any $t\in [t_1, t_2].$
Now, a path $f(\gamma_2)$ is a subpath of
$f(\gamma)=\widetilde{\gamma},$ which belongs to $\Gamma(S(y_0,
1/l), S(y_0, 1/l_0), A(y_0,1/l, 1/l_0)).$ The relation~(\ref{eq3G})
is established.

\medskip
It follows from~(\ref{eq3G}) that
\begin{equation}\label{eq3H}
\Gamma_{k_l}>\Gamma_{f}(S(y_0, 1/l), S(y_0, 1/l_0), A(y_0,1/l,
1/l_0))\,.
\end{equation}
Since $I(\varepsilon, \varepsilon_0)\rightarrow\infty$ as
$\varepsilon\rightarrow 0,$ we may consider that $I(1/l, 1/l_0)>0$
for sufficiently large $l\in {\Bbb N}.$ Set
$$\eta_{l}(t)=\left\{
\begin{array}{rr}
\psi(t)/I(1/l, 1/l_0), & t\in (1/l, 1/l_0)\,,\\
0,  &  t\not\in (1/l, 1/l_0)\,,
\end{array}
\right. $$
where $I(1/l, 1/l_0)=\int\limits_{1/l}^{1/l_0}\,\psi (t)\, dt.$
Observe that
$\int\limits_{1/l}^{\varepsilon_0}\eta_{l}(t)\,dt=1.$ Now, by the
relations~(\ref{eq3.7.2}) and~(\ref{eq3H}), and due to the
definition of $f$ in~(\ref{eq2*B}), we obtain that
$$M_p(\Gamma_{k_l})\leqslant M_p(\Gamma_{f}(S(y_0, 1/l), S(y_0,
1/l_0), A(y_0,1/l,1/l_0)))\leqslant$$
\begin{equation}\label{eq3J}
\leqslant \frac{1}{I^p(1/l, 1/l_0)}\int\limits_{A(y_0, 1/l,
\varepsilon_0)} Q(y)\cdot\psi^{\,p}(|y-y_0|)\,dm(y)\rightarrow
0\quad \text{as}\quad l\rightarrow\infty\,.
\end{equation}
The latter contradicts with~(\ref{eq3A}). The contradiction obtained
above proves the lemma.~$\Box$
\end{proof}

\medskip
The following lemma generalizes Corollary~4.5 in~\cite{Vu}.

\medskip
\begin{lemma}\label{lem2}
{\sl\,Let $n\geqslant 2,$ $n-1<p\leqslant n$ and  let $D$ be a
domain which is unform with a respect to $p$-modulus. Let
$f:D\rightarrow D^{\,\prime}$ be an open discrete and closed mapping
of $D$ onto $D^{\,\prime}$ for which there is a Lebesgue measurable
function $Q:{\Bbb R}^n\rightarrow [0, \infty]$ equals to zero
outside $D^{\,\prime}$ such that the
conditions~(\ref{eq2*B})--(\ref{eqB2}) hold for any point $y_0\in
\partial{D^{\,\prime}}.$ Assume that, there is
$\varepsilon_0=\varepsilon_0(y_0)>0$ and a Lebesgue measurable
function $\psi:(0, \varepsilon_0)\rightarrow [0,\infty]$ such that
the relations~(\ref{eq7***})--(\ref{eq3.7.2}) hold, where $A(y_0,
\varepsilon, \varepsilon_0)$ is defined in~(\ref{eq1**}). Assume
also that, a domain $D$ is locally connected on its boundary, and
that $f$ has a continuous extension
$\overline{f}:\overline{D}\rightarrow \overline{D^{\,\prime}}.$ Then
$\overline{f}$ is light.}
\end{lemma}

\begin{proof}
Assume the contrary, namely, let $y_0\in
\partial D^{\,\prime}$ be some point such that $f^{\,-1}(y_0)\supset K_0,$
where $K_0\subset\partial D$ is some nondegenerate continuum. Then,
in particular, $f(K_0)=y_0.$ Since $\overline{D}$ is a compactum in
$\overline{{\Bbb R}^n}$ and, in addition, $\overline{f}$ is
continuous in $\overline{D},$ the mapping $\overline{f}$ is
uniformly continuous in $\overline{D}.$ In this case,  for any $j\in
{\Bbb N}$ there is $\delta_j<1/j$ such that
$$h(\overline{f}(x),\overline{f}(x_0))$$
\begin{equation}\label{eq3K}
=h(\overline{f}(x),y_0)<1/j \quad \forall\,\, x,x_0\in
\overline{D},\quad h(x, x_0)<\delta_j\,, \quad \delta_j<1/j\,.
\end{equation}
Denote by $B_h(x_0, r)=\{x\in \overline{{\Bbb R}^n}: h(x, x_0)<r\}.$
Then, given $j\in {\Bbb N},$ we set
$$B_j:=\bigcup\limits_{x_0\in K_0}B_h(x_0, \delta_j)\,,\quad j\in {\Bbb N}\,.$$
Since the set $B_j$ is a neighborhood of~$K_0,$
by~\cite[Lemma~2.2]{HK} there is a neighborhood $U_j$ of the set
$K_0$ such that $U_j\subset B_j$ and the set $U_j\cap D$ is
connected. Without loss of generality, we may assume that $U_j$ are
open. Then the set $U_j\cap D$ is path connected, as well
(see~\cite[Proposition~13.1]{MRSY}). Since $K_0$ is a compact set,
there are $z_0, w_0\in K_0$ such that $h(K_0)=h(z_0, w_0).$ It
follows from this, that there are $z_j\in U_j\cap D$ and $w_j\in
U_j\cap D$ such that $z_j\rightarrow z_0$ and $w_j\rightarrow w_0$
as $j\rightarrow\infty.$ We may assume that
\begin{equation}\label{eq2B}
h(z_j, w_j)>h(K_0)/2\quad \forall\,\, j\in {\Bbb N}\,.
\end{equation}
Since the set $U_j\cap D$ is path connected, we may join points
$z_j$ and $w_j$ by some path $\gamma_j$ in $U_j\cap D.$ Set
$C_j:=|\gamma_j|.$

\medskip
Observe that, $h(f(C_j))\rightarrow 0$ as $j\rightarrow\infty.$
Indeed, since $f(C_j)$ is a continuum in $\overline{{\Bbb R}^n},$
there are points $y_j, y^{\,\prime}_j\in f(C_j)$ such that
$h(f(C_j))=h(y_j, y^{\,\prime}_j).$ Then there are $x_j,
x^{\,\prime}_j\in C_j$ such that $y_j=f(x_j)$ and
$y^{\,\prime}_j=f(x^{\,\prime}_j).$ Then points $x_j$ and
$x^{\,\prime}_j$ belong to $U_j\subset B_j.$ Therefore, there are
$x^j_1$ and $x^j_2\in K_0$ such that $x_j\in B(x^j_1, \delta_j)$ and
$x^{\,\prime}_j\in B(x^j_2, \delta_j).$ In this case, by the
relation~(\ref{eq3K}) and due to the triangle inequality we obtain
that
$$h(f(C_j))=h(y_j, y^{\,\prime}_j)=h(f(x_j), f(x^{\,\prime}_j))\leqslant h(f(x_j), f(x^j_1))$$
\begin{equation}\label{eq3L}
+h(f(x^j_1), f(x^j_2))+h(f(x^j_2), f(x^{\,\prime}_j))<2/j\rightarrow
0\quad {\text as}\quad j\rightarrow\infty\,.\end{equation}
It follows from~(\ref{eq2B}) and~(\ref{eq3L}) that, the continua
$C_j,$ $j=1,2,\ldots ,$ satisfy the conditions of Lemma~\ref{lem1}.
By this lemma we may obtain that $h(f(C_j), y_0)\geqslant
\delta_1>0$ for any $j\in {\Bbb N}.$ On the other hand, by the
proving above $x_j\in B(x^j_1, \delta_j).$ Now, by the
relation~(\ref{eq3K}) we obtain that $h(f(x_j),y_0)<1/j,$
$j=1,2,\ldots .$ The resulting contradiction indicates the
incorrectness of the assumption that $\overline{f}$ is not light in
$\partial D.$ Lemma is proved.~$\Box$
\end{proof}

\medskip
\begin{corollary}
{\sl\, The statements of Lemmas~\ref{lem1} and~\ref{lem2} are
fulfilled if we put $D={\Bbb B}^n.$ }
\end{corollary}

\begin{proof}
Obviously, the domain $D={\Bbb B}^n$ is locally connected at its
boundary. We prove that this domain is uniform with respect to
$p$-modulus for $p\in (n-1, n].$ Indeed, since ${\Bbb B}^n$ is a
Loewner space (see~\cite[Example~8.24(a)]{He}), the set ${\Bbb B}^n$
is Ahlfors regular with respect to the Euclidean metric $d$ and
Lebesgue measure in ${\Bbb R}^n$  (see~\cite[Proposition~8.19]{He}).
In addition, in ${\Bbb B}^n,$ $(1; p)$-Poincar\'{e} inequality holds
for any $p\geqslant 1$ (see e.g.~\cite[Theorem~10.5]{HaK}). Now,
by~\cite[Proposition~4.7]{AS} we obtain that the relation
\begin{equation}\label{eq1H}
M_p(\Gamma(E, F, {\Bbb B}^n))\geqslant \frac{1}{C}\min\{{\rm
diam}\,E, {\rm diam}\,F\}\,,
\end{equation}
holds for any $n-1<p\leqslant n$ and for any continua $E, F\subset
{\Bbb B}^n,$ where $C>0$ is some constant, and ${\rm diam}$ denotes
the Euclidean diameter. Since the Euclidean distance is equivalent
to the chordal distance on bounded sets, the uniformity of the
domain $D={\Bbb B}^n$ with respect to the $p$-modulus follows
directly from~(\ref{eq1H}).~$\Box$
\end{proof}

\section{Uniform domains and strongly accessible boundaries}

We need the following statement (see~\cite[Theorem~4.2]{Na}).

\medskip
\begin{proposition}\label{pr1}
{\sl\, Let $\frak{F}$ be a family of connected sets in $D$ such that
$\inf\limits_{F\subset \frak{F}}h(F)>0,$ and let
$\inf\limits_{F\subset \frak{F}}M(\Gamma(F, A, D))>0$ for some
continuum $A\subset D.$ Then
$$\inf\limits_{F, F^{\,*}\subset
\frak{F}}M(\Gamma(F, F^{\,*}, D))>0\,.$$ }
\end{proposition}
Let $p\geqslant 1.$ Due to~\cite[Section~3]{MRSY} we say that a
boundary $D$ is called {\it strongly accessible with respect to
$p$-modulus at $x_0\in \partial D,$} if for any neighborhood $U$ of
the point $x_0\in\partial D$ there is a neighborhood $V\subset U$ of
this point, a compactum $F\subset D$ and a number $\delta>0$ such
that $M_p(\Gamma(E, F, D))\geqslant \delta$ for any continua
$E\subset D$ such that $E\cap \partial U\ne\varnothing\ne E\cap
\partial V.$ The boundary of a domain $D$ is called {\it strongly accessible with respect to
$p$-modulus,} if this is true for any $x_0\in
\partial D.$ When $p=n,$ prefix ''relative to $p$-modulus'' is omitted.
The following lemma is valid (see the statement similar in content
to~\cite[Theorem~6.2]{Na}).

\medskip
\begin{lemma}\label{lem4}
{\sl\, A domain $D\subset {\Bbb R}^n$ has a strongly accessible
boundary if and only if $D$ is uniform.}
\end{lemma}

\begin{proof} The fact that uniform domains have strongly accessible
boundaries has been proved in~\cite[Remark~1]{SevSkv$_1$}. For the
sake of completeness, we present this proof in full. Put $x_0\ne
\infty.$ Let $U$ be a neighborhood of a point $x_0$ and let
$\varepsilon_1>0$ be such that $V:=B(x_0, \varepsilon_1),$
$\overline{V}\subset U.$ Assume that, $\partial U\ne\varnothing$ and
$\partial V\ne\varnothing.$ Let $\varepsilon_2:={\rm
dist}\,(\partial U,
\partial V)>0.$ Let $F$ and $G$ be continua in
$D$ such that $F\cap\partial U\ne\varnothing\ne F\cap\partial V$ and
$G\cap\partial U\ne\varnothing\ne G\cap\partial V.$ It follows from
the last relations that, $h(F)\geqslant \varepsilon_2$ and
$h(G)\geqslant \varepsilon_2.$ Now, by the uniformity of $D$ there
is $\delta=\delta(\varepsilon_2)>0$ such that $M(\Gamma(F, G,
D))\geqslant \delta>0$ for any continua $F$ and $G.$ Thus, $\partial
D$ is strongly accessible at $x_0,$ as required.

\medskip
It remains to prove that domains with strongly accessible boundaries
are uniform.

We will prove this statement from the opposite. Let $D$ be a domain
which has a strongly accessible boundary, but it is not uniform.
Then there is $r>0$ such that, for any $k\in {\Bbb N}$ there are
continua $F_k$ and $F^{\,*}_k\subset D$ such that $h(F_k)\geqslant
r,$ $h(F^{\,*}_k)\geqslant r,$ however,
\begin{equation}\label{eq4A}
M(\Gamma(F_k, F^*_k, D))<1/k\,.
\end{equation}
Let $x_k\in F_k.$ Since $\overline{D}$ is compact set in
$\overline{{\Bbb R}^n},$ we may assume that $x_k\rightarrow x_0\in
\overline{D}.$ Note that the strongly accessibility of the domain
$D$ at the boundary points is assumed to be, and at the inner points
it is even weakly flat, which is the result of V\"{a}is\"{a}l\"{a}'s
lemma (see e.g.~\cite[Sect.~10.12]{Va}, cf.
\cite[Lemma~2.2]{SevSkv$_2$}). Let $U$ be a neighborhood of the
point $x_0$ such that $h(x_0,
\partial U)\leqslant r/2.$ Then there is a neighborhood $V\subset U,$ a compactum $F\subset D$
and a number $\delta>0$ such that the relation $M(\Gamma(E, F,
D))\geqslant \delta$ holds for any continuum $E\subset D$ such that
$E\cap
\partial U\ne\varnothing\ne E\cap
\partial V.$ By the choice of the neighborhood $U,$
we obtain that $F_k\cap U\ne\varnothing\ne F_k\cap (D\setminus U)$
for sufficiently large $k\in {\Bbb N}.$ Observe that, for the same
$k\in {\Bbb N},$ the condition $F_k\cap V\ne\varnothing\ne F_k\cap
(D\setminus V)$ holds. Then, by~\cite[Theorem~1.I.5.46]{Ku} we
obtain that $F_k\cap
\partial U\ne\varnothing\ne F_k\cap \partial V.$ Observe that, a compactum
$F$ can be imbedded in some continuum $A\subset D$
(see~\cite[Lemma~1]{Sm}). Then the inequality $M(\Gamma(E, A,
D))\geqslant \delta$ will only increase. Given the above, we obtain
that
\begin{equation}\label{eq4B}
M(\Gamma(F_k, A, D))\geqslant \delta\qquad \forall\,\,k\geqslant k_0
\end{equation}
for some $k_0\in {\Bbb N}.$ Taking~$\inf$ over all $k\geqslant k_0$
in~(\ref{eq4B}), we obtain that
\begin{equation}\label{eq4C}
\inf\limits_{k\geqslant k_0} M(\Gamma(F_k, A, D))\geqslant \delta\,.
\end{equation}
Set $\frak{F}:=\left\{F_k\right\}_{k=k_0}^{\infty}.$ Now, by the
condition~(\ref{eq4C}) and by Proposition~\ref{pr1}, we obtain that
$\inf\limits_{k\geqslant k_0} M(\Gamma(F_k, F^{\,*}_k, D))>0,$ that
contradicts the assumption made in~(\ref{eq4A}). The resulting
contradiction completes the proof of the lemma.~$\Box$
\end{proof}

\medskip
Obviously, weakly flat boundaries are strongly accessible. Now, by
Lemma~\ref{lem4} we obtain the following.

\medskip
\begin{corollary}\label{cor1}
{\sl\, If $D\subset {\Bbb R}^n$ has a weakly flat boundary, then $D$
is uniform.}
\end{corollary}

\section{The main lemma on lightness. Proof of Theorem~\ref{th2}}

The following statement holds.

\medskip
\begin{lemma}\label{lem7}
{\sl\,Let $n\geqslant 2,$ let $p=n$ and let $D$ be a domain with a
weakly flat boundary. Let $f:D\rightarrow D^{\,\prime}$ be an open
discrete and closed mapping of $D$ onto $D^{\,\prime}$ for which
there is a Lebesgue measurable function $Q:{\Bbb R}^n\rightarrow [0,
\infty]$ equals to zero outside $D^{\,\prime}$ such that the
conditions~(\ref{eq2*B})--(\ref{eqB2}) hold for any point $y_0\in
\partial{D^{\,\prime}}.$ Assume that, there is
$\varepsilon_0=\varepsilon_0(y_0)>0$ and a Lebesgue measurable
function $\psi:(0, \varepsilon_0)\rightarrow [0,\infty]$ such that
the relations~(\ref{eq7***})--(\ref{eq3.7.2}) hold, where $A(y_0,
\varepsilon, \varepsilon_0)$ is defined in~(\ref{eq1**}). Assume
also that, a domain $D^{\,\prime}$ is finitely connected on its
boundary. Then $f$ has a continuous extension
$\overline{f}:\overline{D}\rightarrow \overline{D^{\,\prime}},$
$\overline{f}(\overline{D})=\overline{D^{\,\prime}},$ and it is
light.}
\end{lemma}

\begin{proof}
Obviously, that $Q$ satisfies relations~(\ref{eq7B})--(\ref{eq7C})
of Lemma~\ref{lem5}. Now, by this Lemma, $f$ has a continuous
extension $\overline{f}:\overline{D}\rightarrow
\overline{D^{\,\prime}},$
$\overline{f}(\overline{D})=\overline{D^{\,\prime}}.$ Obviously,
since $D$ has a weakly flat boundary, $D$ has a strongly accessible
boundary, as well. Now, by Lemma~\ref{lem4} $D$ is uniform. Thus,
$\overline{f}$ is light by Lemma~\ref{lem2}.~$\Box$
\end{proof}

\medskip
{\it Proof of Theorem~\ref{th2}.} In the case~1), we choose
$\psi(t)=\frac{1}{t\log\frac{1}{t}},$ and in the case 2), we set
$$\psi(t)\quad=\quad \left \{\begin{array}{rr}
1/[tq^{\frac{1}{n-1}}_{y_0}(t)]\ , & \ t\in (\varepsilon,
\varepsilon_0)\ ,
\\ 0\ ,  &  \ t\notin (\varepsilon,
\varepsilon_0)\ ,
\end{array} \right.$$
Observe that, the relations~(\ref{eq7***})--(\ref{eq3.7.2}) hold for
these functions $\psi,$ where $p=n$ (the proof of this facts may be
found in~\cite[Proof of Theorem~1.1]{Sev$_1$}) Now, the desired
statement follows by Lemma~\ref{lem7}~$\Box$.

\section{The main lemma on discreteness. Proof of Theorem~\ref{th4}}

Finally, we formulate and prove a key statement about the
discreteness of mapping (see~\cite[Theorem~4.7]{Vu}). We note that a
stronger conclusion related to this result is related to a more
stringent condition on the boundary of the mapped domain. Namely, we
require that this domain be locally connected on its boundary.

\medskip
\begin{lemma}\label{lem3}
{\sl\, Let $n\geqslant 2,$ let $p=n$ and let $D$ be a domain with a
weakly flat boundary. Let $f:D\rightarrow {\Bbb R}^n$ be an open
discrete and closed mapping of $D$ for which there is a Lebesgue
measurable function $Q:{\Bbb R}^n\rightarrow [0, \infty]$ equals to
zero outside $D^{\,\prime}$ such that the
conditions~(\ref{eq2*B})--(\ref{eqB2}) hold for any point $y_0\in
\partial{D^{\,\prime}}.$ Assume that, there is
$\varepsilon_0=\varepsilon_0(y_0)>0$ and a Lebesgue measurable
function $\psi:(0, \varepsilon_0)\rightarrow [0,\infty]$ such that
the relations~(\ref{eq7***})--(\ref{eq3.7.2}) hold, where $A(y_0,
\varepsilon, \varepsilon_0)$ is defined in~(\ref{eq1**}). Assume
also that, a domain $D^{\,\prime}$ is locally connected on its
boundary. Then the mapping $f$ has a continuous extension
$\overline{f}:\overline{D}\rightarrow \overline{D^{\,\prime}}$ such
that $N(f, D)=N(f, \overline {D})<\infty.$ In particular,
$\overline{f}$ is discrete in $\overline{D}.$ }
\end{lemma}

\begin{proof}
First of all, the possibility of continuous extension of $f$ to a
mapping $\overline{f}:\overline{D}\rightarrow
\overline{D^{\,\prime}}$ follows by Lemma~\ref{lem5}. Note also that
$N(f, D) <\infty,$ see~\cite[Theorem~2.8]{MS}. Let us to prove that
$N(f, D)=N(f, \overline{D}).$ Next we will reason using the scheme
proof of Theorem~4.7 in \cite{Vu}. Assume the contrary. Then there
are points $y_0\in \partial D^{\,\prime}$ and $x_1,x_2,\ldots, x_k,
x_{k+1}\in
\partial D$ such that $f(x_i)=y_0,$ $i=1,2,\ldots, k+1$ and $k:=N(f,
D).$ We may assume that $y_0\ne\infty.$ Since by the assumption
$D^{\,\prime}$ is locally connected at any point of its boundary,
for any $p\in {\Bbb N}$ there is a neighborhood
$\widetilde{U^{\,\prime}_p}\subset B(y_0, 1/p)$ such that the set
$\widetilde{U^{\,\prime}_p}\cap D^{\,\prime}=U^{\,\prime}_p$ is
connected.

Let us to prove that, for any $i=1,2,\ldots, k+1$ there is a
component $V_p^i$ of the set $f^{\,-1}(U^{\,\prime}_p)$ such that
$x_i\in\overline{V_p^i}.$ Fix $i=1,2,\ldots, k+1.$ By the continuity
of $\overline{f}$ in $\overline{D},$ there is $r_i=r_i(x_i)>0$ such
that $f(B(x_i, r_i)\cap D)\subset U^{\,\prime}_p.$
By~\cite[Lemma~3.15]{MRSY}, a domain with a weakly flat boundary is
locally connected on its boundary. Thus, we may find a neighborhood
$W_i\subset B(x_i, r_i)$ of the point $x_i$ such that $W_i\cap D$ is
connected. Then $W_i\cap D$ belongs to one and only one component
$V^p_i$ of the set $f^{\,-1}(U^{\,\prime}_p),$ while
$x_i\in\overline{W_i\cap D}\subset \overline{V_p^i},$ as required.

Next we show that the sets $\overline{V_p^i}$ are disjoint for any
$i=1,2,\ldots, k+1$ and large enough $p\in {\Bbb N}.$ In turn, we
prove for this that $h(\overline{V_p^i})\rightarrow 0$ as
$p\rightarrow\infty$ for each fixed $i=1,2,\ldots, k+1.$ Let us
assume the opposite. Then there is $1\leqslant i_0\leqslant k+1,$ a
number $r_0>0,$ $r_0<\frac{1}{2}\min\limits_{1\leqslant i,
j\leqslant k+1, i\ne j}h(x_i, x_j)$ and an increasing sequence of
numbers $p_m,$ $m=1,2,\ldots,$ such that $S_h(x_{i_0}, r_0)\cap
\overline{V_{p_m}^{i_0}}\ne\varnothing,$ where $S_h(x_0, r)=\{x\in
\overline{{\Bbb R}^n}: h(x, x_0)=r\}.$ In this case, there are $a_m,
b_m\in V_{p_m}^{i_0}$ such that $a_m\rightarrow x_{i_0}$ as
$m\rightarrow\infty$ and $h(a_m, b_m)\geqslant r_0/2.$ Join the
points $a_m$ and $b_m$ by a path $C_m,$ which entirely belongs to
$V_{p_m}^{i_0}.$ Then $h(|C_m|)\geqslant r_0/2$ for $m=1,2,\ldots .$
On the other hand, since $\partial D$ is strongly accessible, $D$ is
uniform by Lemma~\ref{lem4}, as well. Since $|C_m|\subset
f(V_{p_m}^{i_0})\subset B(y_0, 1/p_m),$ then simultaneously
$h(f(|C_m|))\rightarrow 0$ as $m\rightarrow\infty$ and $h(|C_m|,
y_0)\rightarrow 0$ as $m\rightarrow\infty,$ that contradicts with
Lemma~\ref{lem1}. The resulting contradiction indicates the
incorrectness of the above assumption.

By~\cite[Lemma~3.6]{Vu} $f$ is a mapping of $V_p^i$ onto
$U^{\,\prime}_p$ for any $i=1,2,\ldots, k, k+1.$ Thus, $N(f,
D)\geqslant k+1,$ which contradicts the definition of the number
$k.$ The obtained contradiction refutes the assumption that $N(f,
\overline{D})>N(f, D).$ The lemma is proved.~$\Box$
\end{proof}

\medskip
{\it Proof of Theorem~\ref{th4}.} In the case~1), we choose
$\psi(t)=\frac{1}{t\log\frac{1}{t}},$ and in the case 2), we set
$$\psi(t)\quad=\quad \left \{\begin{array}{rr}
1/[tq^{\frac{1}{n-1}}_{y_0}(t)]\ , & \ t\in (\varepsilon,
\varepsilon_0)\ ,
\\ 0\ ,  &  \ t\notin (\varepsilon,
\varepsilon_0)\ ,
\end{array} \right.$$
Observe that, the relations~(\ref{eq7***})--(\ref{eq3.7.2}) hold for
these functions $\psi,$ where $p=n$ (the proof of this facts may be
found in~\cite[section~3]{Sev$_1$}). The desired conclusion follows
by Lemma~\ref{lem3}.~$\Box$


\medskip
{\bf \noindent Evgeny Sevost'yanov} \\
{\bf 1.} Zhytomyr Ivan Franko State University,  \\
40 Bol'shaya Berdichevskaya Str., 10 008  Zhytomyr, UKRAINE \\
{\bf 2.} Institute of Applied Mathematics and Mechanics\\
of NAS of Ukraine, \\
1 Dobrovol'skogo Str., 84 100 Slavyansk,  UKRAINE\\
esevostyanov2009@gmail.com

\end{document}